\documentclass[12pt]{article}

\usepackage[T1]{fontenc}

\usepackage{amsmath, graphicx, color, amsthm, centernot, amssymb}
\usepackage[font=small,labelfont=bf,width=12.5cm]{caption}
\usepackage{algorithmic, algorithm}
\usepackage{subcaption}
\usepackage{xcolor}
\usepackage{hyperref}

\usepackage{fullpage}

\usepackage{xspace}
\usepackage[obeyFinal,colorinlistoftodos,textsize=tiny]{todonotes}

\newtheorem{claim}{Claim}
\newtheorem{theorem}{Theorem}
\newtheorem{lemma}{Lemma}

\newtheorem{conjecture}{Conjecture}

\newtheorem{problem}{Problem}

\if10     
\usepackage[mathlines]{lineno}
\newcommand*\patchAmsMathEnvironmentForLineno[1]{%
  \expandafter\let\csname old#1\expandafter\endcsname\csname #1\endcsname
  \expandafter\let\csname oldend#1\expandafter\endcsname\csname end#1\endcsname
  \renewenvironment{#1}%
     {\linenomath\csname old#1\endcsname}%
     {\csname oldend#1\endcsname\endlinenomath}}%
\newcommand*\patchBothAmsMathEnvironmentsForLineno[1]{%
  \patchAmsMathEnvironmentForLineno{#1}%
  \patchAmsMathEnvironmentForLineno{#1*}}%
\AtBeginDocument{%
\patchBothAmsMathEnvironmentsForLineno{equation}%
\patchBothAmsMathEnvironmentsForLineno{align}%
\patchBothAmsMathEnvironmentsForLineno{flalign}%
\patchBothAmsMathEnvironmentsForLineno{alignat}%
\patchBothAmsMathEnvironmentsForLineno{gather}%
\patchBothAmsMathEnvironmentsForLineno{multline}%
}
\linenumbers
\fi

\newcommand{\chiC}{\chi_{\mathrm{fum}}}

\usepackage{color}
\definecolor{blue}{rgb}{0,0,1}

\title{Facial unique-maximum colorings of plane graphs with restriction on big vertices}
\author
{
	Bernard Lidick\'y\thanks{Department of Mathematics, Iowa State University, USA.
		E-Mail: \texttt{lidicky@iastate.edu}  } \and		
	Kacy Messerschmidt\thanks{Department of Mathematics, Iowa State University, USA.
		E-Mail: \texttt{kacymess@iastate.edu}} \and	
  	Riste \v{S}krekovski\thanks{Faculty of Information Studies, Novo mesto
  		\& University of Ljubljana, Faculty of Mathematics and Physics
  		\& University of Primorska, FAMNIT, Koper, Slovenia.
   		E-Mail: \texttt{skrekovski@gmail.com}}		
}

\begin{document}
\maketitle

{
	\abstract
	{
		A facial unique-maximum coloring of a plane graph is a proper coloring of the vertices using positive integers 
		such that each face has a unique vertex that receives the maximum color in that face.
		Fabrici and G\"{o}ring (2016) proposed a strengthening of the Four Color Theorem conjecturing that
		all plane graphs have a facial unique-maximum coloring using four colors. This conjecture has been disproven for general
		plane graphs and it was shown that five colors suffice.
		In this paper we show that plane graphs, where vertices of degree at least four induce a star forest, 
		are facially unique-maximum 4-colorable. 
		This improves a previous result for subcubic plane graphs by Andova, Lidick\'y, Lu\v{z}ar, and \v{S}krekovski (2018).
		We conclude the paper by proposing some problems.
	}

}

\section{Introduction}

In this paper, we consider simple graphs only.
A graph is \emph{planar} if it can be drawn in the plane so that no edges cross
and a graph is called \emph{plane} if it is drawn in such a way.
A \emph{coloring} is an assignment of colors to the vertices of a graph.
For the purposes of this paper, we will use integers to represent colors.
A coloring is \emph{proper} if no two adjacent vertices are assigned the same color.
In~\cite{FabGor16}, Fabrici and G\"{o}ring proposed a new type of coloring.
A \emph{facial unique maximum coloring} (or \emph{FUM-coloring}) is a proper coloring using positive integers such that
each face $f$ has only one incident vertex that receives the maximum color on $f$.
For a graph $G$, the minimum number of colors required for a FUM-coloring of $G$
is called the \emph{FUM-chromatic number} of $G$, and is denoted $\chiC(G)$.

One of the most important theorems in graph theory is the Four Color Theorem,
which states that any planar graph has a proper coloring using at most four colors.
Fabrici and G\"{o}ring~\cite{FabGor16} proposed the following strengthening of the Four Color Theorem.

\begin{conjecture}[Fabrici and G\"{o}ring]
	\label{conj:plane4}
	Given any plane graph $G$, $\chiC(G) \leq 4$.
\end{conjecture}

This conjecture was disproven in the general case by the authors~\cite{LidMesSkr17}.
Fabrici and G\"{o}ring \cite{FabGor16} proved that for any plane graph $G$, $\chiC(G) \leq 6$,
while Wendland~\cite{Wen16} improved the upper bound to $5$.
Andova, Lidick\'y, Lu\v{z}ar, and \v{S}krekovski~\cite{AndLidLuzSkr17} proved that if $G$ is a subcubic or outerplane graph, $\chiC(G) \leq 4$.

Recall that a \emph{star} is a connected graph with at most one vertex with degree greater than 1
and a \emph{star forest} is graph consisting of disjoint stars. 
The main result of this paper is the following strengthening of the subcubic result.

\begin{theorem}
	\label{thm:starforest}
	Let $G$ be a plane graph and $X = \{ v \in V(G) : d(v) \geq 4 \}$.
	If $G[X]$ is a star forest, then $\chiC (G) \leq 4$.
\end{theorem}

We present the proof of Theorem~\ref{thm:starforest} in the next section. 
We use the precoloring extension method that was introduced by Thomassen~\cite{Tho94} to show planar graphs are 5-choosable. 
This approach could give even a stronger result with distant crossings~\cite{bib-DLM}.
It could be also used for the list-coloring version~\cite{bib-DLS,bib-thomass3} of Gr\"otzsch theorem, for example.
The approach works by a clever induction. The proof is actually a stronger statement, which allows the induction to work. We state the stronger version of Theorem~\ref{thm:starforest} in the next section.

\section{Proof of the main result}

We prove Theorem~\ref{thm:starforest} as a corollary of a slightly stronger result stated in Lemma~\ref{lem:precolor}.
	
\begin{lemma}
	\label{lem:precolor}
	Let $G$ be a plane graph with a path $P$ on the outer face with at most two vertices.
	Let $X = \{ v \in V(G) : d (v) \geq 4 \} \cup \{ v \in V(P) : d (v) = 3 \}$ such that $G[X]$ is a star forest.
	Given any proper coloring $c$ of the vertices in $P$ using colors $\{ 1, 2, 3 \}$,
	there exists an extension of $c$ to $G$ using colors $\{ 1, 2, 3, 4 \}$ such that color $4$
	does not appear on the outer face and all internal faces have a unique maximum color.
\end{lemma}

\begin{proof}
	Suppose for contradiction that $G$ is a counterexample of minimum order.
	Let $C$ be the boundary walk of the outer face of $G$. 
	
	First we introduce a claim, which we refer to repeatedly when using the minimality of $G$. 
	When we use the minimality of $G$, we color some subgraph $H$ or $G$.
	However, precoloring a vertex of degree 3 in $H$ may ruin the property that $X$ in $H$ is a star forest. The purpose of the following claim is that a vertex can be a newly precolored vertex in $H$ as long as it has degree at least 4 in $H$ or its degree in $G$ is strictly bigger than in $H$.

	\begin{claim}
		\label{claim:subgraph}
		Let $H$ be a proper subgraph of $G$ and $P'$ be a path on the outer face of $H$ with at most two vertices such that
		$d_H (v) < d_G (v)$ for all $v \in V(P') \setminus V(P)$  if $d_H(v) = 3$.
		Then any proper coloring $c'$ of the vertices in $P'$ using colors $\{ 1, 2, 3 \}$ can be extended to a coloring of $H$
		using colors $\{ 1, 2, 3, 4 \}$ such that color $4$ does not appear on the outer face and all internal faces have
		a unique maximum color.
	\end{claim}
	
	\begin{proof}
		Let $X' = \{ v \in V(H) : d_H (v) \geq 4 \} \cup \{ v \in V(P') : d_H (v) = 3 \}$.
		If $X'$ is a star forest, then $H$ can be colored by the minimality of $G$.
		Suppose for contradiction that $X'$ is not a star forest.
		Then there must be a vertex $v$ that is in $X'$ but not in $X$.
		Since $d_H (v) \leq d_G (v)$, we cannot have $d_H (v) \geq 4$,
		otherwise $d_G (v) \geq 4$ and thus $v$ would be in $X$.
		So $v$ must be in the set $\{ v \in V(P') : d_H (v) = 3 \}$.
		This implies that either $d_G (v) \geq 4$ or $v \in P$.
		In either case $v \in X$, which is a contradiction.
		
	\end{proof}
	
	\begin{claim}
		\label{claim:Fcycle}
		$C$ is a cycle.
	\end{claim}
		
	\begin{proof}
	        First note that $G$ has only one connected component incident to the outer face, otherwise Claim~\ref{claim:subgraph} could be used on
	        each such component of $G$ separately.
	        Note that if $G$ has no internal faces, then it is a tree and any proper coloring using $\{1,2,3\}$ works.
	        
	        If every vertex in the outer face is incident to exactly two edges of the outer face, $C$ is a cycle.
		So suppose for contradiction that $G$ has a vertex $v$ incident with at least three edges in the outer face.
		Notice that $v$ is a cut vertex. 
		Let $Y$ be the set of vertices consisting of $v$ and the vertices of the connected component of $G - v$ that intersects $P$,
		if such a component exists. 
		If no component of $G-v$ intersects $P$, pick an arbitrary one. 
		Let $Y'$ be the union of $Y$ and the set of all vertices that are drawn in the interior faces of $G[Y]$ in $G$.
		By the minimality of $G$, there exists a coloring $c_{Y'}$ of $G[Y']$.
		Let $Z = (V(G) \setminus Y') \cup \{ v \}$.
		Since $d_{G[Z]} (v) < d_G (v)$,
		by Claim~\ref{claim:subgraph} a precoloring of $v$ with $c_{Y'} (v)$ can be extended to a coloring $c_Z$ of $G[Z]$.
		Since $v$ has the same color in $c_{Y'}$ and $c_Z$, we can combine these two colorings into a coloring $c$ of $G$, a contradiction.
	\end{proof}

	\begin{claim}
		\label{claim:chord}
		$C$ does not have any chords.
	\end{claim}
	
	\begin{proof}
		Suppose for contradiction that $C$ has a chord $uv$.
		Let $Y$ be the set of vertices consisting of $u$ and $v$ and the vertices of the connected component of $G - \{ u, v \}$ that intersects $P$,
		if such a component exists. 
		If no component of $G-\{u,v\}$ intersects $P$, pick an arbitrary one. 
		Let $Y'$ be the union of $Y$ and the set of all vertices that are drawn in the interior faces of $G[Y]$ in $G$.
		By the minimality of $G$, there exists a coloring $c_{Y'}$ of $G[Y']$.
		Let $Z = (V(G) \setminus Y') \cup \{ u, v \}$.
		Since $d_{G[Z]} (v) < d_G (v)$ and $d_{G[Z]} (u) < d_G (u)$,
		by Claim~\ref{claim:subgraph} a precoloring of $u$ with $c_{Y'} (u)$ and $v$ with $c_{Y'} (v)$
		can be extended to a coloring $c_Z$ of $G[Z]$.
		Since each $u$ and $v$ have the same color in $c_{Y'}$ and $c_Z$,
		we can combine these two colorings into a coloring $c$ of $G$, a contradiction.
		Therefore, $C$ does not have any chords.
	\end{proof}
	
	\begin{claim}
		\label{claim:cycle}
		$G$ is not a cycle.
	\end{claim}
	
	\begin{proof}
		Suppose for contradiction that $G$ is a cycle.
		If a vertex in $P$ receives color $3$, then color the rest of $G$ alternately with colors $1$ and $2$.
		Otherwise, pick a vertex not in $P$ to receive color $3$,
		then color the rest of the vertices alternately with $1$ and $2$.
		The interior face has only one vertex that receives color $3$, hence all interior faces have a unique maximum color, which is a contradiction.
	\end{proof}
	
	\begin{claim}
		\label{claim:23-vertex}
		$C - P$ does not contain a $2$- or $3$-vertex.
	\end{claim}
	
	\begin{figure}[htp!]
		\begin{center}
			\includegraphics[scale=1]{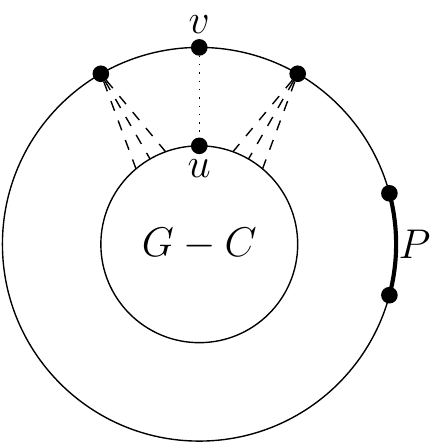}
		\end{center}
		\caption{A non-precolored $2$- or $3$-vertex $v$ on the outer face of $G$}
		\label{fig:23-vertex}
	\end{figure}
	
	\begin{proof}
		Suppose for contradiction that $C - P$ contains a $2$- or $3$-vertex $v$.
		If $v$ is a $2$-vertex, by Claims~\ref{claim:chord} and~\ref{claim:cycle}
		there must be a vertex $u$ that is on the outer face of $G - v$ but not in $C$, as shown in Figure~\ref{fig:23-vertex}.
		If $v$ is a $3$-vertex, let $u$ be the neighbor of $v$ outside of $C$.
		Let $Y = V(G) \setminus \{ u, v \}$.
		By the minimality of $G$, we can color $Y$ such that every internal face of $G[Y]$ has a unique maximum color
		and the vertices in the outer face receive colors from $\{1,2,3\}$.
		Coloring $u$ with $4$ gives a unique maximum color to every face incident with $u$.
		We now color $v$ with the color from $\{ 1, 2, 3 \}$ not used on either of its two neighbors in $C$
		to complete the coloring and arrive at a contradiction.
	\end{proof}
	
	By Claim~\ref{claim:Fcycle}, $C$ cannot have any $1$-vertices.
	By Claim~\ref{claim:23-vertex}, $C - P$ cannot have $2$- or $3$-vertices either.
	Hence each vertex in $C - P$ must have degree at least $4$.
	Since $V(C) \backslash V(P) \subseteq X$, $C - P$ can contain at most three vertices.
	Moreover, since $X$ is acyclic, $P$ is nonempty and contains at least one vertex of degree two.
	
	\begin{claim}
		\label{claim:22444}
		$C - P$ does not contain three vertices.
	\end{claim}
	
	\begin{figure}[htp!]
		\begin{center}
			\includegraphics{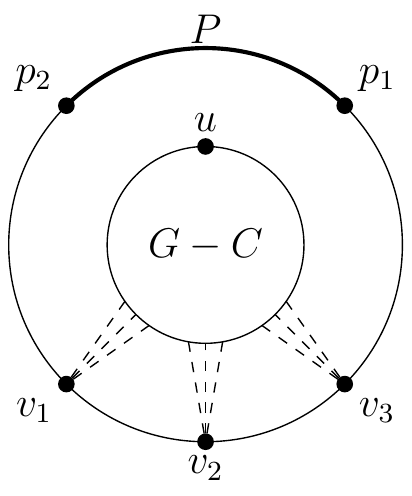}
		\end{center}
		\caption{An outer face consisting of two precolored vertices and three non-precolored vertices of degree at least $4$}
		\label{fig:22444}
	\end{figure}
	
	\begin{proof}
		Suppose for contradiction that $C - P$ contains three vertices.
		Since neither of the vertices in $P$ can belong to $X$, $P$ must consist of one or two $2$-vertices, as shown in Figure~\ref{fig:22444}.
		By Claims~\ref{claim:chord} and~\ref{claim:cycle},
		there must exist a vertex $u$ that is on the outer face of $G - P$ but not in $C$.
		Let $Y = V(G) \setminus \{ p_1, p_2, u \}$.
		Color $v_2$ with the color used on $p_2$ and color $v_3$ with the color from $\{ 1, 2, 3 \}$ not used on $p_1$ or $v_2$.
		By Claim~\ref{claim:subgraph}, we can extend the precoloring on $\{ v_2, v_3 \}$ to all of $G[Y]$
		such that every internal face of $G[Y]$ has a unique maximum color.
		Since $v_2$ received the same color as $p_2$, the color that $v_1$ receives will not conflict with $p_2$ in $G$.
		Coloring $u$ with $4$ gives a unique maximum color to every face incident with $u$,
		thus completing the coloring of $G$ and producing a contradiction.
	\end{proof}
	
	Notice that Claim~\ref{claim:22444} implies that $C$ is either a 4-cycle or a triangle.
	
	\begin{claim}
		\label{claim:2-verticesP}
		$P$ does not consist of two $2$-vertices.
	\end{claim}
	
	\begin{figure}[htp!]
		\begin{center}
			\includegraphics{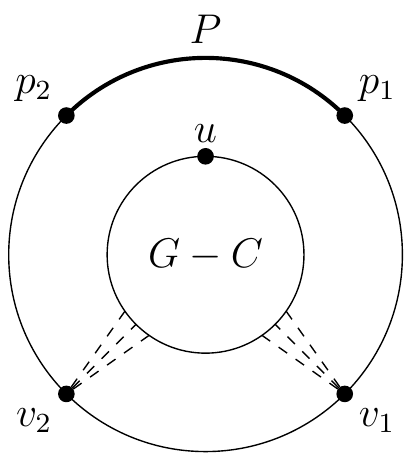}
		\end{center}
		\caption{A precolored path $P$ consisting of two $2$-vertices}
		\label{fig:2-verticesP}
	\end{figure}
	
	\begin{proof}
		Suppose for contradiction that $P$ consists of two $2$-vertices $p_1$ and $p_2$.
		Let $v_1$ and $v_2$ be the neighbors of $p_1$ and $p_2$, respectively, in $C - P$.
		Note that it is possible that $|V(C)| = 3$, in which case $v_1 = v_2$.
		Since $C - P$ has fewer than three vertices, either $v_1 = v_2$ or
		$v_1$ and $v_2$ are the only vertices in $C - P$ and are therefore adjacent along $C$.
		By Claims~\ref{claim:chord} and~\ref{claim:cycle},
		there must exist a vertex $u$ that is on the outer face of $G - P$ but not in $C$, as shown in Figure~\ref{fig:2-verticesP}.
		Let $Y = V(G) \setminus \{ p_1, p_2, u \}$.
		Color $v_1$ with a color from $\{ 1, 2, 3 \}$ not used on $p_1$
		and color $v_2$ with a color from $\{ 1, 2, 3 \}$ not used on either $p_2$ or $v_1$.
		By Claim~\ref{claim:subgraph}, we can extend the precoloring on $\{ v_1, v_2 \}$ to all of $G[Y]$
		such that every internal face of $G[Y]$ has a unique maximum color.
		Coloring $u$ with $4$ gives a unique maximum color to every face incident with $u$,
		thus completing the coloring of $G$ and producing a contradiction.
	\end{proof}
	
	\begin{claim}
		\label{claim:2444}
		$C$ does not have four vertices.
	\end{claim}
	
	\begin{figure}[htp!]
		\begin{center}
			\includegraphics{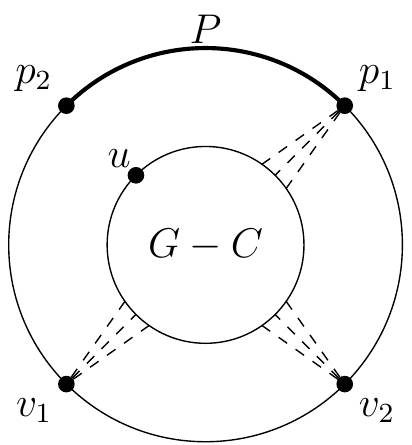}
		\end{center}
		\caption{An outer face consisting of a precolored vertices and two non-precolored vertices of degree at least $4$}
		\label{fig:2444}
	\end{figure}
	
	\begin{proof}
		Suppose for contradiction that $C$ has four vertices.
		Since $C - P$ has at most two vertices, $P$ consists of two vertices, exactly one of which, say $p_2$, is a $2$-vertex, as shown in Figure~\ref{fig:2444}.
		By Claims~\ref{claim:chord} and~\ref{claim:cycle},
		there must exist a vertex $u$ that is on the outer face of $G - P$ but not in $C$.
		Let $Y = V(G) \setminus \{ p_2, u \}$.
		Color $v_2$ with the color used on $p_2$.
		By Claim~\ref{claim:subgraph}, we can extend the precoloring on $\{ p_1, v_2 \}$ to all of $G[Y]$
		such that every internal face of $G[Y]$ has a unique maximum color.
		Since $v_2$ received the same color as $p_2$, the color that $v_1$ receives will not conflict with $p_2$ in $G$.
		Coloring $u$ with $4$ gives a unique maximum color to every face incident with $u$,
		thus completing the coloring of $G$ and producing a contradiction.
	\end{proof}
	
	\begin{claim}
		\label{claim:2-vertexP}
		$P$ does not have a $2$-vertex.
	\end{claim}
	
	\begin{figure}[htp!]
		\begin{center}
			\includegraphics{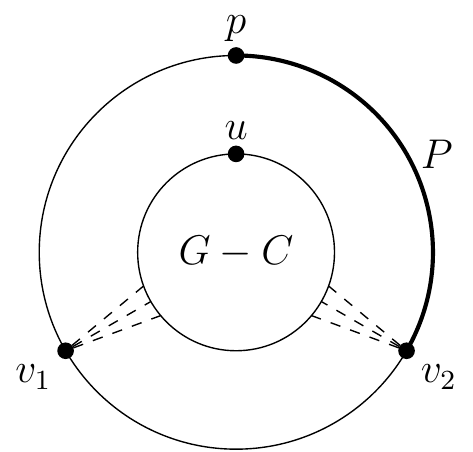}
		\end{center}
		\caption{A precolored $2$-vertex $p$ on the outer face of $G$}
		\label{fig:2-vertexP}
	\end{figure}
	
	\begin{proof}
		Suppose for contradiction that $P$ has a $2$-vertex $p$.
		Let $v_1$ and $v_2$ be the neighbors of $p$ in $C$.
		By symmetry, we can assume that $v_1$ is not in $P$.
		If $|V(P)| = 1$, then $v_2 \in V(C) \setminus V(P)$ and thus has degree at least $4$.
		Since $X$ is a star forest, $v_1$ and $v_2$ must be the only vertices in $C - P$.
		If $|V(P)| = 2$, then $v_2 \in V(P)$ and $d (v_2) \geq 3$.
		Since $d (v_1) \geq 4$, it must be the only vertex in $C - P$.
		In either case, $v_1$ and $v_2$ must be adjacent along $C$.
		If $v_2$ is not in $V(P)$, color $v_2$  with a color from $\{ 1, 2, 3 \}$ not used on $p$.
		Next, color $v_1$ with a color from $\{ 1, 2, 3 \}$ not used on either $p$ or $v_2$.
		By Claims~\ref{claim:chord} and~\ref{claim:cycle}, there must exist a vertex $u$ that is on the outer face of $G - v$ but not in $C$,
		as shown in Figure~\ref{fig:2-vertexP}.
		Let $Y = V(G) \setminus \{ u, p \}$.
		By Claim~\ref{claim:subgraph}, we can extend the precoloring on $\{ v_1, v_2 \}$ to all of $G[Y]$
		such that every internal face of $G[Y]$ has a unique maximum color.
		Coloring $u$ with $4$ gives a unique maximum color to every face of $G$ incident with $u$,
		thus completing the coloring and producing a contradiction.
	\end{proof}
	
	By Claim~\ref{claim:2-vertexP}, each vertex in $P$ must have degree at least $3$.
	Furthermore, by Claim~\ref{claim:23-vertex} each vertex in $C - P$ must have degree at least $4$.
	This means that every vertex in $C$ also belongs to $X$, a contradiction since $X$ is a star forest and $C$ is a cycle.
\end{proof}
	
We are now ready to prove the main theorem of this paper.
	
\begin{proof}[Proof of Theorem~\ref{thm:starforest}]
	Let $G$ be a plane graph and $X = \{ v \in V(G) : d (v) \geq 4 \}$ such that $G[X]$ is a star forest.
	Pick a vertex $v$ on the outer face of $G$ and apply Lemma~\ref{lem:precolor} to $G - v$.
	Color $v$ with $4$ to complete the coloring.
\end{proof}

\section{Conclusion}

Let $G$ be a plane graph and $X = \{ v \in V(G) : d(v) \geq 4 \}$.
In our proof, we used the assumption that $X$ induces a star-forest when arguing before Claim~\ref{claim:22444} that $C$ is short. It might be possible to extend the proof to any acyclic graph.

\begin{conjecture}
	If $G[X]$ is an acyclic graph, then $\chiC (G) \leq 4$.
\end{conjecture}

A special case of Theorem~\ref{thm:starforest} is that if $G[X]$ forms a matching, then  $\chiC (G) \leq 4$.
We believe a stronger result will be true if the maximum degree in $G[X]$ is 2.
If this is true, maybe it could be extended to maximum degree 3 in $G[X]$.
\begin{conjecture}
	If $G[X]$ is a graph of maximum degree 2, then $\chiC (G) \leq 4$.
\end{conjecture}

The connected plane graph $H$ with $\chiC(H) > 4$ found in~\cite{LidMesSkr17} has minimum degree 4 and two vertices of degree five.  
The construction could be disconnected, which gives a 4-regular graph.
It is not clear to us if adding the connectivity constraint gives $\chiC = 4$ or not.

\begin{problem}
	Is there a connected plane graph $G$ with maximum degree 4 with $\chiC (G) > 4$?
\end{problem}

\section{Acknowledgment}

The research project has been supported by the bilateral cooperation between USA and Slovenia, project no. BI--US/17--18--013.
R. \v{S}krekovski was partially supported by the Slovenian Research Agency Program P1--0383.
B. Lidick\'y was partially supported by NSF grant DMS-1600390.

\bibliographystyle{abbrv}
	\bibliography{MainBase}

\end{document}